\documentclass[12pt]{article}

\usepackage{epsfig,amsmath,amsfonts,amssymb,latexsym,color,amsthm,hhline,multirow,subfigure,graphicx}

\newtheorem{theorem}{Theorem}[section]

\newtheorem{prop}{Proposition}[section]
\newtheorem{lemma}{Lemma}[section]

\newcommand{\E}{{\mathbb E}}

\newcommand {\PP}{{\mathbb P}}
\newcommand{\sss}{\scriptscriptstyle}
\newcommand{\Zd}{\mathbb{Z}^d}

\newcommand{\0}{{\bf 0}}
\newcommand{\1}{{\bf 1}}
\newcommand{\cA}{\mathcal{A}}
\newcommand{\cD}{\mathcal{D}}
\newcommand{\vep}{\varepsilon}
\newcommand{\n}{{\bf n}}

\begin{document}

\title{The initial set in the frog model is irrelevant}\parskip=5pt plus1pt minus1pt \parindent=0pt
\author{Maria Deijfen\thanks{Department of Mathematics, Stockholm University; {\tt mia@math.su.se}} \and Sebastian Rosengren \thanks{Department of Mathematics, Stockholm University; {\tt rosengren@math.su.se}}}
\date{December 2019}
\maketitle

\begin{abstract}
\noindent In this note, we consider the frog model on $\Zd$ and a two-type version of it with two types of particles. For the one-type model, we show that the asymptotic shape does not depend on the initially activated set and the configuration there. For the two-type model, we show that the possibility for the types to coexist in that both of them activate infinitely many particles does not depend on the choice of the initially activated sets and the configurations there.
\vspace{0.3cm}


\vspace{0.2cm}

\end{abstract}

\section{Introduction}\label{sec:intro}

The frog model on $\mathbb{Z}^d$ is a growth model driven by random walkers. Each site $x\in \mathbb{Z}^d$ is initially populated with an independent identically distributed (i.i.d.) number $\eta(x)$ of sleeping particles. At time 0 the particles at the origin are activated and start moving according to independent simple symmetric random walks in discrete time. When a site is hit by an active particle, any sleeping particles there are activated and start moving according to independent random walks. The model goes back to  \cite{telcs} and has been further studied e.g.\ in \cite{frogs_shape, frogs_shape_random, phase_transition}. The dynamics can also be based on lazy random walks, where a particle in a given time step independently performs a random walk jump with probability $p$ or stays put with probability $1-p$; see \cite{comp_frogs}.

Recently, a two-type version of the model was introduced in \cite{comp_frogs}, where particles of type $i$ move according to lazy random walks with jump probability $p_i$. All particles in the initial i.i.d.\ configuration are initially sleeping and neutral. At time 0 the particles at the origin are activated and assigned type 1, and the particles at a neighboring site are activated and assigned type 2. Activated particles then start moving and, when a type $i$ particle arrives a new site, any sleeping particles there are activated and assigned type $i$ ($i=1,2$). If two particles of opposite type arrive at a site in the same time step, an arbitrary tie-breaker rule is applied to determine the outcome.

A site is said to be \emph{discovered} when it has been visited by an active particle. One of the main results for the one-type model is a shape theorem for the set of discovered sites, stating that it grows linearly in time and converges to a deterministic shape when scaled by time. In this note, we first observe that the limiting shape does not depend on the choice of the initially active set and the configuration there, that is, the set of discovered sites when an arbitrary configuration of particles in a bounded set $A$ are activated at time 0 converges to the same asymptotic shape as when starting only from the origin (Proposition \ref{prop:shape}). Our main result then concerns the possibility of coexistence in the two-type model, which is said to occur when both types activate infinitely many particles. We show that, if the types can coexist when starting from two bounded disjoint sets $A$ and $B$ populated by arbitrary particle configurations then, for $p_1,p_2\in(0,1)$, they can do so starting from any other two sets $A'$ and $B'$ as well (Theorem \ref{th:coex}). We will not give a literature overview here, but refer to \cite{comp_frogs} for references on related competition models. We mention however that an analog of Theorem \ref{th:coex} for competing first passage percolation was proved in \cite{irrel} using different arguments.\medskip

\noindent \textbf{Definition of the model and notation}

Before stating our results, we give a formal construction of the model with general initial set(s). We give the construction of the two-type model, since the one-type model can be obtained as a special case of this. Let $\nu$ denote the product measure on $\Zd$ defined by the family $\{\eta(x)\}_{x\in \Zd }$ of i.i.d.\ random variables. Fix two bounded disjoint sets $A,B\subset \Zd$ and assign a finite number (random or deterministic) of sleeping particles to each site in $A\cup B$ in an arbitrary way, that is, not necessarily independently and not necessarily according to the same distribution at each site. Then assign sleeping particles to sites in $(A\cup B)^c$ in an i.i.d.\ fashion according to the restriction of $\nu$ to $(A\cup B)^c$. Order the particles at each site and let $(x,j)$ denote particle $j$ at $x\in\Zd$. To each particle $(x,j)$, assign independently a simple symmetric random walk $ \{S^{x,j}_n:\ n\in \mathbb{N} \}$ on $\mathbb{Z}^d$ --- controlling \textit{how} $(x,j)$ moves --- as well as an i.i.d.\ family of delay variables $\{L^{x,j}_{n,k}:\ n,k\in \mathbb{N} \}$ uniformly distributed on $[0,1]$ --- controlling \textit{when} $(x,j)$ moves; see below.

The process is initiated at time 0 in that particles in the sets $A$ and $B$ are activated and assigned type 1 and 2, respectively. Sleeping particles are then activated when their initial site is hit by an active particle. If the site is discovered by a type $i$ particle, the particles are assigned type $i$. If the site is discovered simultaneously by type 1 and 2, the type(s) of its particles is determined by an arbitrary fixed tie-breaker rule. Activated particles move according to their associated random walks and delay variables: A particle $(x,j)$ that has made $n\geq 0$ jumps is located at $S^{x,j}_n$, with $S^{x,j}_0=x$. Assume that the particle is of type $i$ and arrived at $S^{x,j}_n$ at time $t$. Its next move, to $S^{x,j}_{n+1}$, then occurs at time $t+k$ if and only if $L^{x,j}_{n,m}>p_i$ for all $m<k$ and $L^{x,j}_{n,k}\leq p_i$. The particle hence stays at each site for a geometrically distributed number of time steps with parameter $p_i$ and then moves to the next position stipulated by its random walk.

Let $S =\{ (S^{x,j})_{n\in \mathbb{N}}:\ x\in \mathbb{Z}^d,\ j\geq 1\}$ and $L = \{ (L^{x,j}_{n,k})_{n,k\in \mathbb{N}}:\ x\in \mathbb{Z}^d,\ j\geq 1\}$. Write $\Pi^{\sss A,B}$ for a two-type process started from arbitrary particle configurations in the sets $A$ and $B$, as described above, and $\PP^{\sss A,B}$ for the associated probability measure. The full state at time $n$, including the location, type and origin of all particles, is denoted by $\Pi^{\sss A,B}_n$ and the set of discovered sites at time $n$ is denoted by $\xi^{\sss A,B}_n$. Note that $S$ and $L$, together with an initial configuration, can also be used to generate a one-type process started from some arbitrarily populated bounded set $A$ with particles jumping independently with probability $p\in(0,1]$. Such a process and its state at time $n$ is denoted by $\Pi^{\sss A}$ and $\Pi^{\sss A}_n$, respectively, and the set of discovered sites at time $n$ is denoted by $\xi^{\sss A}_n$. A continuum version of the set of discovered sites is given by $\Xi^{\sss A}_n:=\{x+(\frac{1}{2},\frac{1}{2}]:x\in\xi^{\sss A}_n\}$. When the whole particle configuration, including the initial set(s), is drawn from $\nu$, we equip the notation with a wiggle-hat and write $\tilde{\Xi}^{\sss A}_n$, $\tilde{\Pi}^{\sss A}$, $\tilde{\Pi}^{\sss A,B}$ etc.\medskip

\noindent\textbf{Results}

Write $\n=(n,\ldots, 0)$ and consider a one-type process started from the origin $\0$ with the whole initial particle configuration, including the origin, drawn from $\nu$. The set of discovered sites then grows linearly in time. Specifically, for any $p\in (0,1]$ and any $\nu$, conditional on that $\eta(\0)\geq 1$, there exists a non-empty convex set $\mathcal{A}=\mathcal{A}(\nu,p)$ such that for any $\vep \in (0,1)$ almost surely
\begin{equation}\label{eq:shape}
(1-\vep)n\mathcal{A} \subset \tilde{\Xi}^{\textbf{0}}_n \subset (1+\vep)n\mathcal{A}
\end{equation}
for large $n$. This was proved in \cite{frogs_shape} for a non-lazy process with $\eta\equiv 1$ and generalized to other initial distributions in \cite{frogs_shape_random}. The minor additions needed to get the result for a lazy process are described in \cite{comp_frogs}. The shape $\cA$ inherits all symmetries of $\Zd$ and, since the growth occurs in discrete time, $\cA$ cannot exceed the $L_1$ unit-ball.  Apart from this, it is presumably difficult to characterize $\cA$ in general. See however \cite[Theorem 1.2]{frogs_shape} and \cite[Theorem 1.3]{frogs_shape_random} for partial results. Figure \ref{fig} shows a simulation picture of the time-scaled discovered set for $p=1$ when starting with one particle per site.

Our first result is that the set of discovered sites in a one-type process converges to the same shape regardless of the starting set and the configuration there.

\begin{prop}\label{prop:shape}
Consider a one-type process started from an arbitrary non-empty configuration in a bounded set $A$ and with the rest of the particle configuration drawn from $\nu$. For any $p\in (0,1]$, any $\nu$ and any $\vep \in (0,1)$, almost surely
\begin{equation}\label{eq:our_shape}
(1-\vep)n \mathcal{A} \subset \Xi^{\sss A}_n\subset (1+\vep)n \mathcal{A}
\end{equation}
for large $n$, where $\mathcal{A}=\mathcal{A}(\nu,p)$ is the same set as in \eqref{eq:shape}.
\end{prop}

Now consider a two-type process started from $\0$ and $\1$ with the whole initial particle configuration drawn from $\nu$. Write $G_i$ for the event that type $i$ activates infinitely many particles and $C=G_1\cap G_2$ for the event that the types coexist by doing so simultaneously. In \cite{comp_frogs} it is shown that, if either $\eta(x)\geq 1$ almost surely or $\E[\eta(x)]<\infty$ for any $x\in\Zd$, then $C$ has strictly positive probability when $p_1=p_2\in(0,1]$. The condition on $\nu$ is of technical nature, and presumably not necessary for the conclusion. We extend the conclusion to a two-type processes started in an arbitrary way, by showing that the possibility of coexistence does not depend on the choice of the initial sets.

\begin{figure}\label{fig}
	\centering
	\includegraphics[width=7cm, height=7cm]{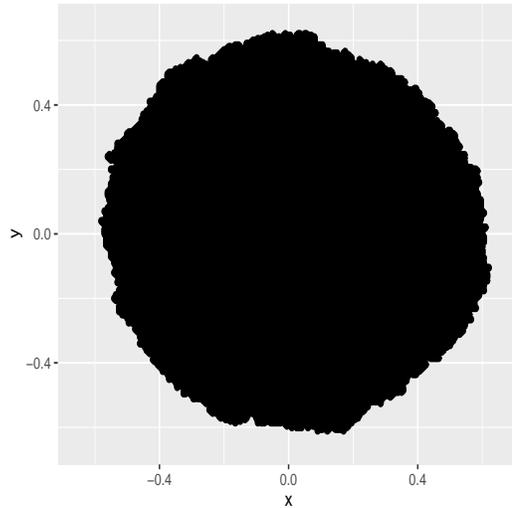}
	\caption{The set $\tilde{\Xi}^\0_n/n$ for $n=500$ with $\eta(x)\equiv 1$ and $p=1$.}
\end{figure}

\begin{theorem}\label{th:coex}
Let $(A,B)$ and $(A',B')$ be two pairs of disjoint bounded subsets of $\Zd$ with arbitrary non-empty particle configurations. For two-type processes started from these sets with the same jump probabilities $p_1,p_2\in(0,1)$, we have that
$$
\PP^{A,B}(C)>0 \Leftrightarrow \PP^{A',B'}(C)>0.
$$
\end{theorem}

The corresponding result with the initial sets restricted to single sites and the whole particle configuration, including the initial sites, drawn from $\nu$ was shown in \cite[Proposition 1.2]{comp_frogs}. Although we expect the conclusion to be true for all values of $p_1$ and $p_2$, neither \cite[Proposition 1.2]{comp_frogs} nor Theorem \ref{th:coex} cover the case when $p_1=1$ or $p_2=1$. This is because the proofs are based on coupling arguments where it is required that particles can stay put in a given time step. See however \cite[Lemma 3.1]{comp_frogs} for partial results for single site initial sets.

Further open problems include studying coexistence issues for unbounded initial sets. Theorem \ref{th:coex} is restricted to bounded initial sets while, for unbounded initial sets, the choice of initial sets could potentially affect the possibility of coexistence; see \cite{unbounded} for related results for competing first passage percolation. For bounded initial sets, the possibility of coexistence is conjectured to be determined by the relation between the one-type shapes $\cA(\nu,p_1)$ and $\cA(\nu,p_2)$ of the two types; see \cite{comp_frogs}. This means that it would be interesting to investigate how laziness affects the shape. When $\nu$ has a very heavy tail, it is known that the shape is unaffected by laziness in the sense that it is given by the $L_1$ unit ball (the maximal shape) for all value of $p_i$; see \cite[Theorem 1.3]{comp_frogs}. However, under suitable conditions on the tail of $\nu$, we conjecture that $\cA(\nu, p_1)$ is a strict subset of $\cA(\nu,p_2)$ when $p_1<p_2$.

We give the short proof of Proposition \ref{prop:shape} in Section 2, followed by the proof of Theorem \ref{th:coex} in Section 3.

\section{Proof of Proposition \ref{prop:shape}}

Let $S_n$ be a simple symmetric random walk on $\Zd$ started from an arbitrary point. It is well-known that the distance from the starting point after $n$ steps scales like $\sqrt{n}$. In particular, for any $\alpha\in(\frac{1}{2},1)$, the probability that the distance exceeds $n^\alpha$ is at most $\exp\{-cn^{2\alpha-1}\}$. This follows from standard results in the theory of moderate deviations and is formulated e.g.\ in \cite[Lemma 2.1]{comp_frogs}. Combining it with the Borel-Cantelli lemma yields the following fact, where $\cD(r)$ denotes the $L_1$-ball with radius $r$, that is, $\cD(r) = \{x\in \mathbb{R}^d:\ ||x||_1 \leq r \}$.

\begin{lemma}\label{le:RW}
For any $\alpha\in(\frac{1}{2},1)$, almost surely $S_n\in\cD(n^\alpha)$ for large $n$.
\end{lemma}

In \cite[Lemma 2.2]{comp_frogs} this is then combined with the shape theorem to conclude that a given particle discovers almost surely finitely many sites. Indeed, since the asymptotic shape $\cA$ is non-empty and convex, we have that $\cA\supset \cD(\delta)$ for some $\delta>0$ and hence the (continuum version of the) set of discovered sites at time $n$ in the one-type model almost surely contains $\cD(n\delta/2)$ for large $n$. Combining this with Lemma \ref{le:RW}, it follows that after some finite time a given particle cannot discover any new sites. A two-type process can be bounded from below by a one-type process consisting of only the slower type, leading to the same conclusion for the two-type model.

\begin{lemma}\label{le:finite}
In both the one-type model $\tilde{\Pi}^{\0}$ and the two-type model $\tilde{\Pi}^{\0,\1}$, the number of sites discovered by a given particle is almost surely finite.
\end{lemma}

\textbf{Remark 2.1.} Note that Lemma \ref{le:finite} so far only applies to processes started with the whole initial configuration drawn from $\nu$ and started from the origin (one-type case) or the origin and a neighbor (two-type case). This is because the proof relies on the (lower bound in) the shape theorem, as described above, and so far we have the shape theorem only for such a process.  Once (the lower bound in) Proposition \ref{prop:shape} is established, we will be able to apply the lemma also to processes $\Pi^{\sss A}$ and $\Pi^{\sss A,B}$ started from arbitrary configurations in arbitrary initial set(s).

\begin{proof}[Proof of Proposition \ref{prop:shape}]
Let $\Sigma$ be a box large enough to contain $\textbf{0}$ and $A$ and couple the processes $\tilde{\Pi}^\0$ and $\Pi^{\sss A}$ so that they are controlled by the same randomness on $\Sigma^c$ but by independent quantities on $\Sigma$. Specifically they have the same initial particle configurations on $\Sigma^c$ and the particles there are controlled by the same random walks and delay variables, while on $\Sigma$ the particle configurations are independent and controlled by independent randomness. Let $N_{\sss \Sigma}$ be the last time in $\tilde{\Pi}^\0$ when a particle originating from $\Sigma$ discovers a new site, where $N_{\sss \Sigma}$ is also taken large enough to ensure that all sites in $\Sigma$ are discovered. By Lemma \ref{le:finite} and \eqref{eq:shape}, this time is almost surely finite. Then consider the process $\Pi^{\sss A}$ and define $N=\min\{n:\xi^{\sss A}_n\supset \tilde{\xi}^\0_{\tilde{N}_{\sss \Sigma}}\}$. We claim that
$$
\tilde{\xi}^\0_{n-N} \subset \xi^{\sss A}_{n}\textrm{ for all } n>N,
$$
that is, if $y\in\tilde{\xi}^\0_{n-N}$ for $n>N$, then $y\in \xi^{\sss A}_{n}$. If $y$ is discovered by a particle originating from $\Sigma$ in $\tilde{\Pi}^\0$ this is clear, since all such sites belong to $\xi^{\sss A}_N$ by the definition of $N$. If $y$ is discovered by a particle in $\Sigma^c$ in $\tilde{\Pi}^\0$, we note that all such particles will be further along their random walk trajectories in $\Pi^{\sss A}_n$ ($n>N$) compared to $\tilde{\Pi}^\0_{n-N}$ -- this follows from the fact that $\tilde{\xi}^\0_N\subset \xi^{\sss A}_N$ and the definition of $N$, and gives the desired implication. Now, pick $\vep \in (0,1)$ and $\delta\in(0,\vep)$. Using \eqref{eq:shape}, we get that
\begin{equation}\label{eq:shape_lower}
(1-\vep)n\cA \subset (1-\delta)(n-N)\cA \subset \tilde{\xi}^\0_{n-N} \subset \xi^{\sss A}_{n},
\end{equation}
almost surely for large $n$.

The second inclusion in \eqref{eq:our_shape} follows from a similar argument. Let $N'_{\sss \Sigma}$ be the last time in $\Pi^A$ when a particle originating from $\Sigma$ discovers a new site, where $N'_{\sss \Sigma}$ is also taken large enough to ensure that all sites in $\Sigma$ are discovered. Since we now have a lower linear bound \eqref{eq:shape_lower} for $\xi^{\sss A}_{n}$, it follows from Lemma \ref{le:finite} and Remark 2.1 that this time is almost surely finite. Define $N'=\min\{n:\tilde{\xi}^\0_n\supset \xi^A_{N'_{\sss \Sigma}}\}$. By the same argument as above, we have that $\xi^{\sss A}_n \subset \tilde{\xi}^\textbf{0}_{N'+n}$ and hence, for any $\vep>0$, we get from \eqref{eq:shape} that
\begin{equation}
\xi^A_{n} \subset \tilde{\xi}^\textbf{0}_{N'+n}\subset (1+\epsilon)n\mathcal{A}
\end{equation}
almost surely for large $n$.
\end{proof}

\section{Proof of Theorem \ref{th:coex}}
		
We first observe the simple fact that the discovered set in a two-type process can be bounded from below by the one-type shape of the slower type. Here we include the jump probabilities in the notation for the discovered set.

\begin{lemma}\label{le:2tlb}
Fix $\nu$ and $p_1\leq p_2\in(0,1]$. Then, for any initial sets $(A,B)$ and any $\vep\in (0,1)$, we have that $\Xi^{\sss A,B}_n(p_1,p_2)\supset (1-\vep)n\cA(\nu,p_1)$ almost surely for large $n$.
\end{lemma}

\begin{proof}
Consider a one-type process $\Pi^{A\cup B}$ with the same initial particle configuration as the two-type process, constructed based on the same random elements $S$ and $L$, and with jump probability $p_1(\leq p_2)$. By the construction of the model, we have that $\xi^{\sss A\cup B}_n(p_1)\subset \xi^{\sss A,B}_n(p_1,p_2)$ and the statement then follows from Proposition \ref{prop:shape}.
\end{proof}

\begin{proof}[Proof of Theorem \ref{th:coex}]
We show that, if $\PP^{\sss A,B}(C)>0$, then $\PP^{\sss A',B'}(C)>0$ as well. To this end, as in the proof of Proposition \ref{prop:shape}, let $\Sigma$ be a box large enough to contain $A\cup B\cup A'\cup B'$ and define the following times for the process started from $(A,B)$:

\begin{itemize}
\item[-] Let $M_{\sss \Sigma}$ be the last time in $\Pi^{\sss A,B}$ when a particle with initial location in $\Sigma$ discovers a new site, where $M_{\sss \Sigma}$ is also taken large enough to ensure that all sites in $\Sigma$ are discovered. It follows from Lemma \ref{le:finite}, Remark 2.1 and Proposition \ref{prop:shape} that $M_{\sss \Sigma}<\infty$ almost surely.
\item[-] Let $M_{\sss \cA}$ be such that $\Xi^{\sss A,B}_n(p_1,p_2)\supset \frac{n}{2}\cA(\nu,p_1)$ for $n\geq M_{\sss \cA}$, and note that $M_{\sss \cA}<\infty$ almost surely by Lemma \ref{le:2tlb}.
\end{itemize}

Take $k$ such that $\frac{n}{2}\cA(\nu,p_1)\supset \cD(n^{3/4})$ for $n>k$, and define $M=\max\{M_{\sss \Sigma},M_{\sss \cA},k\}$. Since $M<\infty$ almost surely, we have for large $m$ that
\begin{equation}\label{eq:CM}
\PP^{\sss A,B}(C\cap \{M\leq m\})\geq \frac{1}{2}\PP^{\sss A,B}(C)>0.
\end{equation}

Write $\Pi'(\Sigma)$ for an independent particle configuration on $\Sigma$, distributed as the initial state of a process started from $(A',B')$, with associated independent random walks and delay variables, and let $M'$ be such that, if all particles in $\Pi'(\Sigma)$ start moving at time 0, then none of them is located outside $\cD(n^{3/4})$ at time $n$ for $n>M'$. Then $M'<\infty$ by Lemma \ref{le:RW}. Fix $m$ large to ensure both \eqref{eq:CM} and $\PP(M'\leq m)\geq 1/2$.

Now pick random quantities -- initial configuration, random walks $S$ and delay variables $L$ -- such that both types activate infinitely many particles in $\Pi^{\sss A,B}$ and such that $M\leq m$ in that process. Also pick $\Pi'(\Sigma)$ with $M'\leq m$. Consider a process started from the sets $(A',B')$, with $\Sigma$ initially populated as in $\Pi'(\Sigma)$ and $\Sigma^c$ initially populated as in $\Pi^{\sss A,B}$. We claim that, by controlling the movements of a finite number of particles in the beginning of the time course, after some finite time we can obtain a configuration where (i) the status and location of all particles originating from $\Sigma^c$ is the same as in $\Pi^{\sss A,B}_m$, and where (ii) all particles originating from $\Sigma$ are activated and located at their initial position. This is achieved e.g.\ by the following scenario:

\begin{itemize}
\item [1.] First one of the initially active type 1 particles from $A'$ activates all sleeping particles in $\Sigma$. The activated particles stay at their initial sites, and so do the type 2 particle(s) in $B'$ and any additional type 1 particles in $A'$.
\item[2.] A type 1 particle from $A'$ and a type 2 particle from $B'$ then move together and activate all particles originating from sites in $\xi^{\sss A,B}_m\cap \Sigma^c$ with the same type as in $\Pi^{\sss A,B}_m$. The two particles then return to their initial positions. All other particles in $\Sigma$, as well as the activated particles in $\Sigma^c$, stay at their initial sites.
\item[3.] The activated particles in $\Sigma^c$ finally move to their positions in $\Pi^{\sss A,B}_m$, while the particles in $\Sigma$ stay at their initial locations.
\end{itemize}

Let $\tau$ denote the minimal time required to achieve this. We call this the coupling time. At this point, we couple the process to the randomness in $\Pi^{\sss A,B}$ and $\Pi'(\Sigma)$ by letting all particles in $\Sigma^c$ move according to the same random objects that control their behavior after time $m$ in the process $\Pi^{\sss A,B}$, while all particles in $\Sigma$ move according to the same random objects that control their behavior from time 0 in $\Pi'(\Sigma)$. We claim that, with this construction, all particles in $\Sigma^c$ are activated by the same type as in $\Pi^{\sss A,B}$ so that, in particular, both types activate infinitely many particles if they do so in $\Pi^{\sss A,B}$. To see this, we need to see that the fact that the configuration in $\Sigma$ is different compared to $\Pi_m^{\sss A,B}$ does not affect the activation of sites in $\Sigma^c$. First note that, since $n>M_{\sss \Sigma}$, the fact that there may be fewer particles in $\Sigma$ in the coupled process compared to $\Pi_m^{\sss A,B}$ will not cause fewer sites to be activated. Specifically, for any $t>0$, the set of discovered sites at time $\tau+t$ in the coupled process contains $\Pi^{\sss A,B}_{m+t}$. That any additional particles in $\Sigma$ in the coupled process will not discover any new sites then follows from that $m>M'$, $m>M_{\sss \cA}$ and $m>k$. Indeed, the set of discovered sites in the coupled process at time $\tau+t$ will exceed $\frac{m+t}{2}\cA(\nu,p_1)$, while the particles with initial location in $\Sigma$ will never reach outside $\cD((m+t)^{3/4})$.

Write $C^{\sss A,B}$ for the event that both types activate infinitely many sites in the original process $\Pi^{\sss A,B}$ and $\hat{C}^{\sss A',B'}$ for the same event in the coupled process started from $(A',B')$, that is, in a process generated by independent randomness up to time $\tau$ and then coupled to $\Pi^{A,B}$ and $\Pi'(\Sigma)$ as described above. Also let $\hat{F}$ denote the event that the scenario described in 1-3 occurs in the time interval $[0,\tau]$ in the coupled process.  We then have that
\begin{eqnarray*}
\PP(\hat{C}^{\sss A',B'}) & \geq & \PP(\hat{C}^{\sss A',B'}|C^{\sss A,B}\cap \{M\leq m\}\cap \{M'\leq m\})\cdot \PP(C^{\sss A,B}\cap \{M\leq m\}\cap \{M'\leq m\})\\
& \geq & \PP(\hat{F})\cdot \PP(C^{\sss A,B}\cap \{M\leq m\})\cdot\PP(M'\leq m),
\end{eqnarray*}
where we have used the fact that $M'$ is independent of the randomness in $\Pi^{\sss A,B}$. Here $\PP(\hat{F})>0$ since the scenario in 1-3 is obtained by controlling the movements of a finite number of particles during a finite time interval, and the last two factors are positive by the choice of $m$. Hence $\PP(\hat{C}^{A',B'})>0$ and, since the coupled processes has the same distribution $\PP^{A',B'}$ as an original process started from $(A',B')$, we obtain that $\PP^{\sss A',B'}(C)>0$, as desired.
\end{proof}


\begin{thebibliography}{99}

\bibitem{frogs_shape} Alves, O., Machado, F. and Popov, S. (2002): The shape theorem for the frog model, \emph{Ann. Appl. Probab.} \textbf{12}, 533-546.

\bibitem{frogs_shape_random} Alves, O., Machado, F., Popov, S. and Ravishankar, K. (2001): The shape theorem for the frog model with random initial configuration, \emph{Markov Proc. Rel. Fields} \textbf{7}, 525-539.

\bibitem{phase_transition} Alves, O., Machado, F. and Popov, S. (2002): Phase transition for the frog model, \emph{Electr. J. Probab.} \textbf{7}, 1-21.

\bibitem{irrel} Deijfen, M. and H\"aggstr\"om, O. (2006): The initial configuration is irrelevant for the possibility of mutual unbounded growth in the two-type Richardson model, {\em Comb. Probab. Computing} \textbf{15}, 345-353.

\bibitem{unbounded} The two-type Richardson model with unbounded initial configurations, {\em Ann. Appl. Probab.} \textbf{17}, 1639-1656.

\bibitem{comp_frogs} Deijfen, M.\, Hirscher, T.\ and Lopes, F.\ (2019): Competing frogs on $\Zd$, \emph{Electr. J. Probab.}, to appear.

\bibitem{drift} D\"{o}bler, C., Gantert, N.,  H\"{o}felsauer, T., Popov,S. and Weidner, F.\ (2018): Recurrence and transience of frogs with drift on $\Zd$, \emph{Electron. J. Probab.} \textbf{23}, 1-23.

\bibitem{telcs} Telcs, A. and Wormald, N. (1999): Branching and tree indexed random walks on fractals, \emph{J. Appl. Probab.} \textbf{36}, 999-1011.

\end{thebibliography}
\end{document}